\documentclass[12pt]{article}
\usepackage{amsfonts}
\usepackage{amsmath,amsthm}
\usepackage{cases}

\newtheorem{lemma}{Lemma}[section]

\newtheorem{theorem}{Theorem}[section]
\newtheorem{proposition}[theorem]{Proposition}
\theoremstyle{definition}
\newtheorem{definition}{Definition}[section]
\newtheorem{corollary}[theorem]{Corollary}
\newtheorem{remark}{Remark}

\newtheorem{claim}{Claim}
\newcommand{\comment}[1]{}

\numberwithin{equation}{section}
\begin{document}
\title{Li-Yau inequality for unbounded Laplacian on graphs}
\author{Chao Gong, Yong Lin, Shuang Liu, Shing-Tung Yau}
\date{}
\maketitle

\begin{center}
\textbf{Abstract}
\end{center}
In this paper, we derive Li-Yau inequality for unbounded Laplacian on complete weighted graphs with the assumption of the curvature-dimension inequality $CDE'(n,K)$, which can be regarded as a notion of curvature on graphs. Furthermore, we obtain some applications of Li-Yau inequality, including Harnack inequality, heat kernel bounds and Cheng's eigenvalue estimate. These are first kind of results on this direction for unbounded Laplacian on graphs.
\section{Introduction}
Li-Yau inequality is a powerful tool for studying positive solutions to the heat equation on manifolds. The simplest version of Li-Yau inequality states
\begin{equation}\label{eq:ly0}
\frac{|\nabla u|^2}{u^2} - \frac{\partial_t u}{u} \leq \frac{n}{2t},~~t>0
\end{equation}
where $u$ is a positive solution of the heat equation \mbox{$(\Delta - \partial_t)u = 0$} on an $n$-dimensional compact Riemannian manifold with non-negative Ricci curvature, see \cite{LY86}. After Li and Yau's breakthrough in 1986, great efforts were made to establish an analogue result on different settings. In 2006, Bakry and Ledoux generalized Li-Yau inequality to measure spaces for diffusion Laplace operator that satisfy chain rule by using the curvature-dimension
inequality ($CD$ condition). It is obvious that chain rule for Laplace operator fails on graphs.

However, on graphs, in 2015, Bauer et.al. proved a discrete vision of Li-Yau inequality which is very similar to the original one \eqref{eq:ly0}, as follows: let $u$ be a positive solution to the heat equation on the graph $G$, then
\begin{equation}\label{eq:ly1}
\frac{\Gamma(\sqrt{u})}{u}-\frac{\partial_t \sqrt{u}}{\sqrt{u}}\leq\frac{n}{2t},~~t>0.
\end{equation}

This version of gradient estimate is sharp on graph setting(see \cite{BHLLMY15}).

In order to bypass the chain rule, the authors modified the curvature-dimension condition, which is called exponential curvature-dimension inequality ($CDE'(n,0)$ condition), see \cite{BHLLMY15}. After that, using the heat semigroup technique, Horn et.al. proved the same Li-Yau inequality as \eqref{eq:ly1} of bounded and positive solution of heat equation on graphs, see \cite{HLLY17}. Li-Yau inequality and its applications on
graphs has been studied by many authors, we refer to \cite{BHY14,M14,LL16,Q17}. All these results are established for bounded Laplacian. However, we know that the classic results of Li-Yau on Riemannian manifolds were for the unbounded Laplace-Beltrami operator.

Studies on unbounded Laplacian seem much more difficult on graphs and the proofs are usually different with the bounded cases. There are some results involving the curvature-dimension inequality for unbounded Laplacian. In \cite{HL17}, the authors derived stochastic completeness of complete graphs by proving an equivalent property of $CD(\infty, 0)$ condition, named gradient estimate. In \cite{LMP16}, the authors obtained diameter bound by utilizing gradient estimate on non-negative curved graphs.
Other equivalent properties of $CD(n,K)$, such as Poincar\'{e} inequality and reverse Poincar\'{e} inequality condition,  were proved in \cite{GL17}. After that, in \cite{H}, the author derived the Liouville theorem for bounded Harmonic functions on non-negative curved graphs using reverse Poincar\'{e} inequality. However, there is no result about Li-Yau inequality for unbounded Laplacian.

In this paper, we study unbounded Laplacian, and prove Li-Yau inequality for unbounded Laplacian on infinite graphs under the exponential curvature-dimension
condition. Moreover we obtain some applications, including Harnack inequality, heat kernel upper bounds and Cheng's eigenvalue estimate.

Let us introduce the setting and then state our main results.

\subsection{Setting}
Let $G=(V,E)$ be an infinite graph with the set of vertices $V$ and the set of edges $E$, a symmetric subsets of $V\times V$. Two vertices are called neighbours if they are connected by an edge $\{x,y\}\in E$, which is denoted by $x\sim y$. At a vertex $x$, if $\{x, x\}\in E$, we say there is a
loop at $x$. In this paper, we do allow loops for graphs. We say the graph is connected if for any distinct $x,y\in V$ there is a finite path such that $x=x_0\sim x_1\sim\cdots \sim x_n=y$. In this paper, we just consider connected graphs.

On $(V,E)$, we assign a measure on vertices by a function $m:V\rightarrow \mathbb{R}^+$,
and give a weight on edges by a function $\omega: E\rightarrow \mathbb{R}^+$, the edge $\{x,y\}\in E$ has weight $\omega_{xy}>0$, and the weight function is symmetric, i.e. $\omega_{xy}=\omega_{yx}$. We call the quadruple $G=(V,E,m,\omega)$ a weighted graph. In this paper, we restrict our interest to the locally finite graph, that is
\[\deg(x):=\sum_{y\sim x}\omega_{xy}<\infty,~~~~\forall x\in V.\]

We denote by $V^\mathbb{R}$ the set of real-value functions on $V$, and
by $C_0(V)$ the set of finitely supported functions on $V$. we denote by
$\ell^p_m, p\in [1,\infty]$ the $\ell^p_m$ spaces of functions on $V$ with respect to the measure $m$, and by $\|\cdot\|_{\ell^p_m}$ the $p$-norm of a function. For any $f,g\in \ell^2_m$, we let $\langle f,g\rangle=\sum_{x\in V}f(x)g(x)m(x)$ denote the standard inner product. This makes $\ell^2_m$ a Hilbert space.

To a weighted graph $G$, it associates with a Dirichlet form w.r.t $\ell_m^2$,
\[\begin{split}
Q:D(Q)\times D(Q)&\rightarrow \mathbb{R}\\
f&\mapsto Q(f):=\frac{1}{2}\sum_{x,y\in V}\omega_{xy}(f(y)-f(x))(g(y)-g(x)),
\end{split}\]
where $D(Q)$ is defined as the completion of $C_0(V)$ under
the $Q$-norm $\|\cdot\|_Q$ given by
\[\|f\|_Q=\sqrt{\|f\|_{\ell^2_m}^2+\frac{1}{2}\sum_{x,y\in V}\omega_{xy}(f(y)-f(x))^2}.\]
we refer from \cite{KL12}. For locally finite graphs, the associated generator $\Delta$, called Laplacian defined by
\[\Delta f(x)=\frac{1}{m(x)}\sum_{y\sim x}\omega_{xy}(f(y)-f(x)), ~~~~f\in D(\Delta),\]
where $D(\Delta)=\{f\in D(Q)|\Delta f\in \ell_m^2\}$. In this paper, we restrict all functions on
\[\mathcal{F}:=\{f\in V^\mathbb{R}:\sum_{y\sim x}\omega_{xy}|f(y)|<\infty~\mbox{for all}~x\in V\}.\]
The Laplacian $\Delta$ generates a semigroup $P_t f=e^{t\Delta}f$ acting on $\ell^p_m$ for $p\in[1,\infty]$. Obviously, the measure $m$ plays an important role in the definition of Laplacian. Given the weight $\omega$ on $E$, there are two typical choices of Laplacian as follows:
\begin{itemize}
  \item $m(x)=\deg(x)$ for all $x\in V$, which is called the normalized graph Laplacian;
  \item $m(x)\equiv 1$ for all $x\in V$, which is the combinatorial graph Laplacian.
\end{itemize}
Note that normalized Laplacian is bounded. Actually the fact that the Laplacian $\Delta$ is bounded on $\ell^p_m$ is equivalent to the following condition:
\begin{equation}\label{eq:bounded}
\sup_{x\in V}\frac{\deg(x)}{m(x)}<\infty,
\end{equation}
see \cite{KL12}. As we mentioned before, Li-Yau inequality and its applications were well studied for bounded Laplacian. Thus, in this paper, we are interested in combinatorial graph and more general unbounded Laplacians.

In this paper, we further assume the measure $m$ on $V$ is \textit{non-degenerate}, i.e.
\begin{equation}\label{eq:non-degenerate}
\inf\limits_{x\in V}m(x)=\delta>0.
\end{equation}
This assumption yields a very useful fact for $\ell^p_m$
spaces, see \cite{HL17}.
\begin{lemma}\label{lemma1.1}
Let $m$ be a non-degenerate measure on $V$ as \eqref{eq:non-degenerate}. Then
for any $1 \leq p < q \leq \infty, \ell^p_m\hookrightarrow \ell^q_m.$
\end{lemma}

Now we introduce the gradient forms associated to the Laplacian and curvature dimension conditions on graphs following \cite{LY10,BHLLMY15}.
\begin{definition}
The gradient forms $\Gamma$ and the iterated gradient form $\Gamma_2$ are defined by
\[\begin{split}
2\Gamma(f,g)(x)&=(\Delta(fg)-f\Delta(g)-g\Delta(f))(x)\\
&=\frac{1}{m(x)}\sum_{y\sim x}\omega_{xy}(f(y)-f(x))(g(y)-g(x)),
\end{split}\]
\[2\Gamma_2(f,g)(x)=(\Delta\Gamma(f,g)-\Gamma(f,\Delta(g))-\Gamma(g,\Delta(f)))(x).\]
We write $\Gamma(f)=\Gamma(f,f),\Gamma_2(f)=\Gamma_2(f,f)$ for simplification.
\end{definition}
\begin{definition}
The graph $G$ satisfies the CD inequality $CD(n,K)$ if, for any function $f\in V^\mathbb{R}$ and at every vertex $x \in V$
\begin{equation*}
\Gamma_{2}(f)\geq \frac{1}{n}(\Delta f)^{2}+K\Gamma(f). \label{eqn:cd}
\end{equation*}
\end{definition}
On graphs, the
$CD$ condition implies a weak Harnack-type inequality (see [4]), but it seems insufficient to prove the
Li-Yau inequality. However, in \cite{BHLLMY15,HLLY17}, the authors proved a discrete analogue of Li-Yau inequality by modifying the curvature notion, which is called $CDE'$ condition. In the following, we recall their definition.

\begin{definition}
We say that a graph $G$ satisfies the $CDE'(x,n,K)$, if for any positive function $f \in V^\mathbb{R}$, we have
\begin{equation}
\widetilde{\Gamma_2}(f)(x) \geq \frac{1}{n} f(x)^2\left(\Delta \log f\right)(x)^2 + K \Gamma(f)(x).
\label{eqn:cde'}
\end{equation}
We say that $CDE'(n,K)$ is satisfied if $CDE'(x,n,K)$ is satisfied for all $x \in V$.
\end{definition}
We introduce a relation between the $CD$ conditon and the $CDE'$ condition.
\begin{remark}\label{rem:cde}
$CDE'(n,K)$ implies $CD(n,K)$ on graphs but visa versa is not true(see \cite{M15}). For diffusion Laplace operator, for example the Laplace-Beltrami operator
on Riemannian manifolds, the $CDE'(n,K)$ is equivalent to $CD(n,K)$(see \cite{BHLLMY15}).
\end{remark}

Next, we introduce a condition for the
completeness of infinite weighted graphs:
The graph $G$ is called \textit{complete}, that is, there exists a non-decreasing sequence $\{\eta_k\}_{k=0}^{\infty}\in C_0(V)$ such that
\begin{equation}\label{eq:complete}
\lim\limits_{k\rightarrow\infty}\eta_k=\mathbf{1}, ~~ \mbox{and}~~\Gamma(\eta_k)\leq\frac{1}{k},
\end{equation}
where $\mathbf{1}$ is the constant function on $V$, and its limit is pointwise.
This condition was defined for Markov diffusion
semigroups in \cite{BGL14} and adopted to graphs in \cite{HL17, GL17,GL}.
It is been proved that a large class of graphs under the assumptions possessing appropriate intrinsic metrics have been
shown to be complete, see \cite{HL17}. In particular, graphs with bounded Laplacians are always complete. This lemma shows that $C_0(V)$ is a dense subset of $D(Q)$(see \cite{HL17}).
\begin{lemma}\label{lemma1.2}
Let $G=(V,E,m,\omega)$ be a complete graph. for any $f\in D(Q)$, we have
$$\left\|f\eta_k-f\right\|_Q\rightarrow 0,\quad k\rightarrow\infty$$
\end{lemma}

\subsection{Main results}
The first main result is the following Li-Yau inequality.
\begin{theorem}\label{th:ly-family}
Let $G=(V,E,m,\omega)$ be a complete graph and $m$ be a non-degenerate measure. If $G$ satisfies $CDE'(n,K)$,
then for any $0\leq f\in \ell^p_m$ with $p\in [1,\infty]$, we have
\begin{equation}\label{eq:ly-family2}
\frac{\Gamma(\sqrt{P_tf})}{P_tf}\leq\frac{1}{2}\left(1-\frac{2Kt}{2b+1}\right)\frac{\Delta P_tf}{P_tf}+\frac{n}{2}\left(\frac{b^2}{(2b-1)t}+\frac{K^2 t}{2b+1}-K\right).
\end{equation}
\end{theorem}
\begin{remark}
When $K=0$ and $b=1$, let $0\leq f\in \ell^\infty_m$, then $u=e^{t\Delta}f$ solves the heat equation, i.e. $\Delta u=\partial_t u$ with $0< u\in\ell^\infty_m$, this family of Li-Yau inequality \eqref{eq:ly-family2} reduced to the familiar Li-Yau inequality
\begin{equation}\label{eq:ly}
\frac{\Gamma(\sqrt{u})}{u}-\frac{\partial_t \sqrt{u}}{\sqrt{u}}\leq\frac{n}{2t},~~t>0.
\end{equation}
\end{remark}
The first application of Li-Yau inequality is Harnack inequality. We find boundedness of Laplacian is not required in the proof of Harnack inequality from Li-Yau inequality in \cite{BHLLMY15}. Integrating the Li-Yau inequality of the positive solution of the heat equation \eqref{eq:ly1} over $t$, we have the following corollary under the assumption of $CDE'(n,0)$.
In the context of deriving Harnack inequality and heat kernel upper bounds, two additional assumptions are needed as follow:
\[\omega_{\min}:=\inf_{x,y\in V,x\sim y}\omega_{xy}>0,\]
and
\[m_{\max}:=\sup_{x\in V}m(x)<\infty,\]

\begin{corollary}\label{coro:hanack}
Suppose $G=(V,E,m,\omega)$ be a complete graph and $m$ be a non-degenerate measure. If $G$ satisfies $CDE'(n,0)$,
then for all $x,z\in V$ and any $t < s$, and any $0\leq f\in \ell^p_m$ with $p\in [1,\infty]$, one has
\begin{equation}\label{eq:hanack}
P_t f(x)\leq  P_s f(z)\left(\frac{s}{t}\right)^{n} \exp\left(\frac{4m_{\max} d(x,z)^{2}}{\omega_{\min}( s-t)}\right),
\end{equation}
where $d(x,z)$ is the graph distance, that is, the number of edge of the shortest path from $x$ to $z$.
\end{corollary}
As one of the most important applications of Li-Yau inequality, Harnack inequality can be derived to the following heat kernel upper bounds in these assumptions. We can define $p(t,x,y)=P_t\left(\frac{\delta_y}{m(y)}\right)(x)$ as the heat kernel (the fundamental solution of the heat equation) on weighted graph $G$, where $\delta_y(x)$ equals to $1$ when $x=y$, otherwise equals to $0$. We denote by $B(x, r) = \{y\in V: d(x,y)\leq r\}$ the ball centered in $x$ with radius $r$
, and denote by  $V(A):=\sum_{x\in A}m(x)$ the volume of a subset $A$ of $V$, we will write $V (x,r)$ for $V (B(x, r))$.
\begin{theorem}\label{th:Gauss-upper}
Suppose $G=(V,E,m,\omega)$ be a complete graph and $m$ be a non-degenerate measure. If $G$ satisfies $CDE'(n,0)$,  then there exist constant
$C>0$ depending on $n$ so that, for any $x,y\in V$ and for all $t>0$,
\[ p(t,x,y)\leq \frac{C}{V(x,\sqrt{t})}.\]
\end{theorem}

We say that a weighted graph G satisfies the assumption $(A)$ if one of the following holds:\\
$(A_1)$ The Laplacian $\Delta$ is bounded on $\ell^2_m$, i.e.  \eqref{eq:bounded} holds.\\
$(A_2)$ $G$ is complete, that is \eqref{eq:complete} holds, and $m$ is non-degenerate, see \eqref{eq:non-degenerate}.


As a further application of the Li-Yau inequality, we obtain an estimate
for the greatest lower bound of the $\ell^2$-spectrum known as Cheng's eigenvalue
estimate.
\begin{theorem}\label{th:Cheng}
Let G be a finite or locally finite graph satisfying $(A)$ and the $CDE'(n,-K)$ condition with some $K>0$,  and let $\lambda^*$ be the greatest lower bound for the $\ell^2$-spectrum of the graph Laplacian $\Delta$. Then we have
\[\lambda^* \leq \frac{Kn}{2}.\]
\end{theorem}
\begin{remark}
The upper bounds of eigenvalues presented here are stronger than the ones with the slight weak curvature condition of Bauer et.al in \cite{BHLLMY15}.
\end{remark}
The paper is organized as follows. In next section, we introduce some useful propositions of heat semigroup. We prove some uniform estimations 
based on a discrete Caccioppoli inequality for Poisson's equations. In section 3, we prove our main results: the family of Li-Yau inequality for unbounded Laplacian with the assumption of $CDE'(n,K)$, and its applications, including Harnack inequality, heat kernel upper bounds and Cheng's eigenvalue estimate on graphs.

\section{Preliminaries}

The following integration by parts formula is useful later, see  \cite{KL10}.
\begin{lemma}[Green's formula]\label{lem:green}
For any $f\in D(Q)$ and $g\in D(\Delta)$ we have
$$\sum\limits_{x\in V}f(x)\Delta g(x)m(x)=-\sum\limits_{x\in V}\Gamma(f,g)(x)m(x).$$
\end{lemma}

The next proposition is a consequence of standard Dirichlet form theory,
see \cite{FOT11} and \cite{KL12}.

\begin{proposition}\label{pro:semi}
For any $f\in \ell^p_m,p\in[1,\infty]$, we have $P_tf\in\ell^p_m$ and
\[\|P_tf\|_{\ell^p_m}\leq\|f\|_{\ell^p_m},~~~~\forall t\geq0.\]
Moreover, $P_tf\in D(\Delta)$ for any $f\in \ell^2_m$.
\end{proposition}

Now, we introduce some useful properties of heat semigroup $P_t$ on graphs.

\begin{proposition} For any $t,s>0,$ we have
\begin{enumerate}
  \item $P_t$ and $\Delta$ are communicative, i.e. for any $f\in D(\Delta)$,
\[\Delta P_tf=P_t\Delta f.\]
  \item $P_t$ satisfies the semigroup property. That is, for any $f\in \ell^p_m$,
\[P_{t}\circ P_sf=P_{t+s}f.\]
  \item $P_t$ is self-adjoint. That is,  for any $f,g \in \ell^2_m$,
\[\langle P_t f,g\rangle=\langle P_t g,f\rangle.\]
\end{enumerate}
\end{proposition}
We introduce the Caccioppoli inequality for subsolutions to Poisson¡¯s equations on graphs, see \cite{HL17}. By proving the Caccioppoli inequality for subsolutions to Poisson¡¯s equations, the authors can derive a uniform upper bound about $P_tf$, see (\cite{HL17}, Lemma 3.6).
\begin{lemma}\label{lem:Cac}
Let $g,h\in V^\mathbb{R}$, if the graph $G$ satisfies
$$\Delta g\geq h$$
Then for any $\eta\in C_0(V)$, we have
$$\left\|\Gamma(g)\eta^2\right\|_{l^1_m}\leq C\left(\left\|\Gamma(\eta)g^2\right\|_{l^1_m}+\left\|gh\eta^2\right\|_{l^1_m}\right).$$
\end{lemma}
\begin{lemma}\label{lem:0}
Let $G=(V,E,m,\omega)$ be a complete graph, and $m$ be a non-degenerate measure. For any $f\in C_0(V)$ and $T>0$, we have $\max\limits_{[0,T]}\Gamma(P_tf)\in \ell_m^1$, moreover there exists $C_1(T,f)>0$ such that
\begin{equation}\label{eq:Gamma}
\left\|\max\limits_{[0,T]}\Gamma(P_tf)\right\|_{\ell_m^1}\leq C_1(T,f).
\end{equation}
In addition, we have $\max_{[0,T]}|\Gamma(P_tf,\Delta P_tf)|\in \ell_m^1$, moreover there exists $C_2(T,f)>0$ such that
\begin{equation}\label{eq:GammaDelta}
\left\|\max\limits_{[0,T]}|\Gamma(P_tf,\Delta P_tf)|\right\|_{\ell^1_m}\leq C_2(T,f).
\end{equation}
\end{lemma}

\begin{lemma} \label{lem:2}
For any functions $f,g\in V^\mathbb{R}$, if $|f|\leq B_1$, $|g|\geq B_2>0$, then there exist positive constants $C_1(B_1),C_2(B_1,B_2)>0$ such that $$\left|\Gamma(\frac{f}{g})\right|\leq C_1(B_1)\left|\Gamma(f)\right|+C_2(B_1,B_2)\left|\Gamma(g)\right|.$$
\end{lemma}
\begin{proof}
\begin{equation*}
\left.
\begin{aligned}
\left|\Gamma(\frac{f}{g})\right|
&=\frac{1}{2}\sum\limits_{y\sim x}\omega_{xy}(\frac{f(y)}{g(y)}-\frac{f(x)}{g(x)})^2\\
&=\frac{1}{2}\sum\limits_{y\sim x}\omega_{xy}[\frac{1}{g(y)}(f(y)-f(x))+f(x)(\frac{1}{g(y)}-\frac{1}{g(x)})]^2\\
&\leq2\frac{1}{2}\sum\limits_{y\sim x}\omega_{xy}\frac{1}{g^2(y)}(f(y)-f(x))^2\\
&+2\frac{1}{2}\sum\limits_{y\sim x}\omega_{xy}\frac{f^2(x)}{g^2(x)g^2(y)}(g(y)-g(x))^2\\
&\leq\frac{2}{B_1^2}\left|\Gamma(f)\right|+\frac{2B_2^2}{B_1^4}\left|\Gamma(g)\right|\\
&=C_1(B_1)\left|\Gamma(f)\right|+C_2(B_1,B_2)\left|\Gamma(g)\right|
\end{aligned}
\right.
\end{equation*}
\end{proof}
From the above two Lemmas, the authors in \cite{GL} have proved the following result, the key is to plus a positive constant $\epsilon$ on $P_t f$, which is still a solution of the heat equation, and makes it has a lower bound. We copy the proof here for sake of completeness.

From now, we let $u(t,x)=P_tf(x)+\epsilon$, where $\epsilon>0$ is a constant.
\begin{lemma}\label{lem:1} Let $G=(V,E,m,\omega)$ be a complete graph, and $m$ be a non-degenerate measure. For any $0\leq f\in C_0(V)$ and $T>0$,  then there exists a positive constant $C_1(\epsilon,T,f)>0$ such that
\begin{equation}\label{eq:GammaSqrt}
\left\|\max_{t\in [0,T]}\Gamma(\sqrt{u})\right\|_{l_m^1}\leq C_1(\epsilon,T,f).
\end{equation}
Moreover, there exists a positive constant $C_2(\epsilon,T,f)>0$ such that
\begin{equation}\label{eq:GammaSqrtFrac}
\left\|\max_{t\in [0,T]}\Gamma\left(\sqrt{u},\frac{\Delta u}{2\sqrt{u}}\right)\right\|_{l^1_m}\leq C_2(\epsilon,T,f).
\end{equation}
\end{lemma}
\begin{proof}
For any $0\leq f\in C_0(V)$, we have $\epsilon\leq P_tf+\epsilon\leq\|f\|_\infty+\epsilon$, by \eqref{eq:Gamma} in Lemma \ref{lem:0},
\begin{equation*}
\begin{aligned}
\left\|\max_{[0,T]}\Gamma(\sqrt{P_tf+\epsilon})\right\|_{l_m^1}
&=\frac{1}{2}\sum\limits_{x\in V}\max_{[0,T]}\sum\limits_{y\sim x}\omega_{xy}(\sqrt{P_tf+\epsilon}(y)-\sqrt{P_tf+\epsilon}(x))^2\\
&=\frac{1}{2}\sum\limits_{x\in V}\max_{[0,T]}\sum\limits_{y\sim x}\omega_{xy}(\frac{P_tf(y)-P_tf(x)}{\sqrt{P_tf+\epsilon}(y)+\sqrt{P_tf+\epsilon}(x)})^2\\
&\leq\frac{1}{4\epsilon}\left(\frac{1}{2}\sum\limits_{x\in V}\max_{[0,T]}\sum_{y\sim x}\omega_{xy}((P_tf)(y)-(P_tf)(x))^2\right)\\
&\leq \frac{1}{4\epsilon}\left\|\max\limits_{[0,T]}\Gamma(P_tf)\right\|_{l_m^1}\\
&\leq \frac{1}{4\epsilon}C_1(T,f)=:C_1(\epsilon,T,f).
\end{aligned}
\end{equation*}
Due to $\Delta f\in C_0(V)$, and from Lemma \ref{lem:0}, we have
$$\left|\Gamma(\Delta (P_{t}f+\epsilon))\right|(x)=\left|\Gamma(\Delta P_{t}f)\right|(x)=\left|\Gamma(P_{t}\Delta f)\right|(x)\leq \max\limits_{t\in[0,T]}\left|\Gamma(P_{t}\Delta f)\right|(x)\in \ell^1_m,$$
moreover there exists a constant $C_1(T,\Delta f)$ such that
\[\left\|\Gamma(\Delta u)\right\|_{l^1_m}\leq C_1(T,\Delta f).\]
By Cauchy-Schwartz inequality, we have
$$\Gamma(f,g)\leq\sqrt{\Gamma(f)\Gamma(g)}\leq\frac{1}{2}(\Gamma(f)+\Gamma(g)).$$
Combining with Lemma \ref{lem:2}, we obtain
\[\begin{aligned}
\left\|\Gamma\left(\sqrt{u},\frac{\Delta u}{2\sqrt{u}}\right)\right\|_{l^1_m}
&\leq\frac{1}{2}\left\|\Gamma(\sqrt{u})\right\|_{l^1_m}+ \frac{1}{8}\left\|\Gamma\left(\frac{\Delta u}{\sqrt{u}}\right)\right\|_{l^1_m}\\
&\leq \left(\frac{1}{2}+C_1(\epsilon,f)\right)\left\|\Gamma(\sqrt{u})\right\|_{l^1_m}+\frac{1}{8}C_2(\epsilon,f)\left\|\Gamma(\Delta u)\right\|_{l^1_m}\\
&\leq \left(\frac{1}{2}+C_1(\epsilon,f)\right)C_1(\epsilon,T,f)+\frac{1}{8}C_2(\epsilon,f)C_1(T,\Delta f)=:C_2(\epsilon,T,f).
\end{aligned}\]
That finishes the proof.
\end{proof}

We will also need to prove the following two technique Lemmas.
\begin{lemma} \label{lem:6}
Let $G=(V,E,m,\omega)$ be a complete graph and $m$ be a non-degenerate measure. For any $0\leq f\in C_0(V)$ and $t\in [0,T]$, let $\epsilon>0$, then we have $(\Delta \log u)^2\in \ell^1_m$.
Moreover, there exists a positive constant $C_3(T,f,\epsilon)>0$ such that
\[\left\|\max_{[0,T]}(\Delta \log u)^2\right\|_{\ell^1_m}\leq C_3(T,f,\epsilon).\]
\end{lemma}
\begin{proof}
Since $\log z\leq z-1$ for any $z\geq0$, then for any $x\in V$, we have
\begin{equation*}
\Delta \log u(x)=\frac{1}{m(x)}\sum_{y\sim x}\omega_{xy}\log\frac{u(y)}{u(x)}\leq\frac{\Delta u(x)}{u(x)},
\end{equation*}
For any $x\in V$ such that $\Delta \log u(x)\geq 0$ , it is true that
\begin{equation}\label{eq:estimate1}
\left(\Delta \log u(x)\right)^2\leq\frac{1}{\epsilon^2}(\Delta u)^2 (x)\leq\frac{2}{\epsilon^2}(\Delta u)^2 (x).
\end{equation}
On the other hand,
we have
\[\begin{split}
\Delta \log u(x)
&=-\frac{1}{m(x)}\sum_{y\sim x}\omega_{xy}\log \frac{u(x)}{u(y)}\\
&\geq\frac{1}{m(x)}\sum_{y\sim x}\omega_{xy}\frac{u(y)-u(x)}{u(y)}\\
&=-\frac{1}{m(x)}\sum_{y\sim x}\omega_{xy}\frac{(u(y)-u(x))^2}{u(x)u(y)}+\frac{\Delta u(x)}{u(x)}\\
&\geq-\frac{1}{\epsilon^2}\Gamma(u)(x)+\frac{\Delta u(x)}{u(x)}
\end{split}\]
then for any $x\in V$ such that $\Delta \log u(x)< 0$, we have
\begin{equation}\label{eq:estimate2}
\begin{split}
(\Delta \log u(x))^2
&\leq\frac{2}{\epsilon^4}(\Gamma(u)(x))^2+\frac{2}{\epsilon^2}(\Delta u(x))^2.
\end{split}
\end{equation}

%

Then from \eqref{eq:estimate1} and \eqref{eq:estimate2}, we obtain
\[\sum_{x\in V}(\Delta \log u)^2(x)m(x)\leq\frac{2}{\epsilon^2}\sum_{x\in V}(\Delta u(x))^2 m(x)+\frac{2}{\epsilon^4}\sum_{x\in V,\Delta \log u(x)< 0}(\Gamma(u)(x))^2m(x).\]

From Proposition \ref{pro:semi}, we know that $P_t f\in D(\Delta)$ with $f\in C_0(V)\subset \ell^2_m$. Moreover, due to $\Delta P_t f=P_t\Delta f$ and $\Delta f\in C_0(V)$ when $f\in C_0(V)$, we have $\Delta P_t f \in D(Q)$ by \eqref{eq:Gamma} in Lemma \ref{lem:0}. Therefore, from Green formula, we have
\[
\frac{2}{\epsilon^2}\sum_{x\in V}m(x)(\Delta u)^2(x)=\frac{2}{\epsilon^2}\sum_{x\in V}m(x)(\Delta P_t f)^2(x)
=-\frac{2}{\epsilon^2}\sum_{x\in V}m(x)\Gamma(P_t f,\Delta P_t f)(x).
\]
Since the measure is non-degenerate, i.e. $\inf_{x\in V}m(x)=\delta>0$, then we have
\[\sum_{x\in V,\Delta \log u(x)< 0}(\Gamma(u)(x))^2m(x)\leq\sum_{x\in V}(\Gamma(u)(x))^2m(x)\leq \frac{1}{\delta}\left(\sum_{x\in V}\Gamma(u)(x)m(x)\right)^2.\]
According to Lemma \ref{lem:0}, we can conclude what we desire.
\end{proof}

Since
\begin{equation}\label{eq:solution}
\Delta\sqrt{u}=\frac{\Delta u}{2\sqrt{u}}-\frac{\Gamma(\sqrt{u})}{\sqrt{u}}.
\end{equation}

Then $CDE'(n,K)$ condition for $\sqrt{u}$ can be rewritten into
\begin{equation}\label{eq:CDE'}
\frac{1}{2}\Delta\Gamma(\sqrt{u})-\Gamma(\sqrt{u},\frac{\Delta u}{2\sqrt{u}})\geq \frac{1}{n}u(\Delta\log \sqrt{u})^2+K\Gamma(\sqrt{u}).
\end{equation}

Then we have the following Lemma.
\begin{lemma} \label{lem:important}
 Let $G=(V,E,m,\omega)$ be a complete graph and the measure $m$ is non-degenerate. If $G$ satisfies $CDE'(n,K)$ with $K\in \mathbb{R}$, then for any $0\leq f\in C_0(V)$ and $t\geq 0$, let $\epsilon>0$, we have
$$\Gamma(\sqrt{u})\in D(Q).$$
\end{lemma}
\begin{proof}
We know $\Gamma(\sqrt{u})\in \ell^1_m$, then by Lemma \ref{lemma1.1}, we have $\Gamma(\sqrt{u})\in \ell^2_m, \ell^{\infty}_m$. From Theorem 6 in \cite{KL12}, we know that $\Gamma(\sqrt{u})$ can be approximated in $Q$-norm by functions with finite support. Then we just need to show that $$Q(\Gamma(\sqrt{u}))<\infty.$$

Coming back to the Caccioppoli inequality, we let $g=\Gamma(\sqrt{u})$, $h=2\Gamma(\sqrt{u},\frac{\Delta u}{2\sqrt{u}})+\frac{ 2}{n}u(\Delta\log \sqrt{u})^2+2K\Gamma(\sqrt{u}),$ from $CDE'(n,K)$, we know that
\[\Delta g\geq h.\]
Let $\{\eta_k\}_0^{\infty}$ be a non-decreasing sequence, then from the Caccioppoli inequality, we have
\begin{equation*}
\begin{aligned}
\left\|\Gamma(g)\eta_k^2\right\|_{l^1_m}
&\leq C\left(\left\|\Gamma(\eta_k)g^2\right\|_{l^1_m}+\left\|gh\eta_k^2\right\|_{l^1_m}\right)\\
&\leq C\left(\frac{1}{k}\left\|g\right\|_{l^2_m}^2+\left\|g\right\|_{l^{\infty}_m}\left\|\Gamma\left(\sqrt{u},\frac{\Delta u}{2\sqrt{u}}\right)\right\|_{l^1_m}+\left\|g\right\|_{l^{\infty}_m}\|u\|_{\ell^\infty_m}\left\|(\Delta\log u)^2\right\|_{l^1_m}+\left|2K\right|\left\|g\right\|_{l^2_m}^2\right).\\
\end{aligned}
\end{equation*}
From Lemma \ref{lem:1} and Lemma \ref{lem:6}, and notice that $g\in l^2(V,m)$ and $g\in l^{\infty}(V,m)$, we have
$$\left\|\Gamma(g)\eta_k^2\right\|_{l^1_m}\leq C.$$
According to Fatou lemma,
$$\left\|\Gamma\left(\Gamma(\sqrt{P_tf+\epsilon})\right)\right\|_{l^1_m}\leq\lim_{k\rightarrow\infty}\inf\left\|\Gamma\left(\Gamma(\sqrt{P_tf+\epsilon})\right)\eta_k^2\right\|_{l^1_m}\leq C.$$
This proves the lemma.
\end{proof}

\section{Li-Yau inequality}
For any $f,\xi\in C_0(V)$, let
$$\Phi(s):=\sum\limits_{x\in V}\Gamma\left(\sqrt{P_{t-s}f+\epsilon}\right)P_s\xi(x)m(x),\quad \epsilon>0.$$

\begin{proposition}
$\Phi(s)$ is differentiable in $s \in (0, t)$, and
\begin{equation}
\begin{aligned}
\Phi'(s)&=-2\sum\limits_{x\in V}\Gamma\left(\sqrt{P_{t-s}f+\epsilon},\frac{\Delta( P_{t-s}f+\epsilon)}{2\sqrt{P_{t-s}f+\epsilon}}\right)(x)P_s\xi(x)m(x)\\
&+\sum\limits_{x\in V}\Gamma(\sqrt{P_{t-s}f+\epsilon})(x)\Delta P_s\xi(x)m(x)
\end{aligned}
\end{equation}
\end{proposition}
\begin{proof}
Without loss of generality, we assume that $s\in (\tau,t-\tau)$ for some $\tau>0$. Taking the formal derivative of $\Phi(s)$ in $s$, we obtain
\[
-2\sum\limits_{x\in V}\Gamma\left(\sqrt{P_{t-s}f+\epsilon},\frac{\Delta( P_{t-s}f+\epsilon)}{2\sqrt{P_{t-s}f+\epsilon}}\right)(x)P_s\xi(x)m(x)+\sum\limits_{x\in V}\Gamma(\sqrt{P_{t-s}f+\epsilon})(x)\Delta P_s\xi(x)m(x).\]
If one can show that the absolute values of summands are uniformly (in s) controlled by summable functions, then this formal derivative is the the derivative of $\Phi(s)$. Note that
\[\|P_s\xi\|_{\ell^\infty}\leq\|\xi\|_{\ell^\infty}<\infty.\]
Combining with the inequality \eqref{eq:GammaSqrtFrac} in Lemma \ref{lem:1}, we have
\begin{equation*}
\begin{aligned}
&2\left|\Gamma\left(\sqrt{P_{t-s}f+\epsilon},\frac{\Delta( P_{t-s}f+\epsilon)}{2\sqrt{P_{t-s}f+\epsilon}}\right)(x)\right|P_s\xi(x)\\
&\leq \sup_{s\in (\tau,t-\tau)}2\left|\Gamma\left(\sqrt{P_{t-s}f+\epsilon},\frac{\Delta( P_{t-s}f+\epsilon)}{2\sqrt{P_{t-s}f+\epsilon}}\right)(x)\right|P_s\xi(x)\\
&\leq 2\|\xi\|_{\ell^\infty}\sup_{s\in (\tau,t-\tau)}\left|\Gamma\left(\sqrt{P_{t-s}f+\epsilon},\frac{\Delta( P_{t-s}f+\epsilon)}{2\sqrt{P_{t-s}f+\epsilon}}\right)(x)\right|=:g(x)\in \ell^1_m.
\end{aligned}
\end{equation*}

In addition, due to the non-degenerate measure $m$, every finitely supported function $\xi$ lies in the domain of Laplacain, and thus $\Delta \xi\in \ell^2_m\subset\ell^\infty_m$. Then we have
$$|\Delta P_s\xi|=|P_s\Delta \xi|\leq\left\|P_s\Delta \xi\right\|_{l^{\infty}_m}\leq\left\|\Delta \xi\right\|_{l^{\infty}_m}<\infty,$$
therefore, the inequality \eqref{eq:GammaSqrt} in Lemma \ref{lem:1} yields that
\[\begin{aligned}
\Gamma(\sqrt{P_{t-s}f+\epsilon})(x)\left|\Delta P_s\xi(x)\right|&\leq\sup_{s\in (\tau,t-\tau)}\Gamma(\sqrt{P_{t-s}f+\epsilon})(x)\left|\Delta P_s\xi(x)\right|\\
&\leq\left\|\Delta f\right\|_{l^{\infty}_m}\sup_{s\in (\tau,t-\tau)}\Gamma(\sqrt{P_{t-s}f+\epsilon})(x)=:h(x)\in\ell^1_m.
\end{aligned}\]
\end{proof}
Since $P_s\xi\in D(\Delta), \Gamma(\sqrt{P_{t-s}f+\epsilon})\in D(Q)$, by Green formula (see Lemma \ref{lem:green}),
we have
\begin{equation}\label{eq:derive}
\begin{aligned}
\Phi'(s)&=-2\sum\limits_{x\in V}\Gamma\left(\sqrt{P_{t-s}f+\epsilon},\frac{\Delta( P_{t-s}f+\epsilon)}{2\sqrt{P_{t-s}f+\epsilon}}\right)(x)P_s\xi(x)m(x)\\
&-\sum\limits_{x\in V}\Gamma\left(\Gamma(\sqrt{P_tf+\epsilon}),P_s\xi\right)(x)m(x)
\end{aligned}
\end{equation}
\begin{claim} \label{cl:1}
Assume the complete graph $G=(V,E,m,\omega)$ satisfying $CDE'(n,K)$, then for any  $0\leq h\in D(Q)$, we have
\begin{equation}
\begin{aligned}
&-2\sum\limits_{x\in V}\Gamma(\sqrt{P_{t-s}f+\epsilon},\frac{\Delta( P_{t-s}f+\epsilon)}{2\sqrt{P_{t-s}f+\epsilon}})(x)h(x)m(x)\\
&-\sum\limits_{x\in V}\Gamma(\Gamma(\sqrt{P_tf+\epsilon}),h(x))m(x)\\
&\geq \frac{2}{n}\sum\limits_{x\in V}(P_{t-s}f+\epsilon)(\Delta\log\sqrt{P_{t-s}f+\epsilon})^2h(x)m(x)+2K\sum\limits_{x\in V}\Gamma(\sqrt{P_{t-s}f+\epsilon})h(x)m(x)
\end{aligned}
\end{equation}
\end{claim}
\begin{proof}
In order to prove the above claim, first we consider $0\leq h\in C_0(V)$.  By Green's formula and the definition of $CDE'(n,K)$,
\begin{equation*}
\left.
\begin{aligned}
&\quad -2\sum\limits_{x\in V}\Gamma(\sqrt{P_{t-s}f+\epsilon},\frac{\Delta( P_{t-s}f+\epsilon)}{\sqrt{P_{t-s}f+\epsilon}})(x)h(x)m(x)\\
&-\sum\limits_{x\in V}\Gamma(\Gamma(\sqrt{P_tf+\epsilon}),h(x))m(x)\\
&=-2\sum\limits_{x\in V}\Gamma(\sqrt{P_{t-s}f+\epsilon},\frac{\Delta( P_{t-s}f+\epsilon)}{\sqrt{P_{t-s}f+\epsilon}})(x)h(x)m(x)\\
&+\sum\limits_{x\in V}\Delta\Gamma(\sqrt{P_{t-s}f+\epsilon})h(x)m(x)\\
&\geq \frac{2}{n}\sum\limits_{x\in V}(P_{t-s}f+\epsilon)(\Delta\log\sqrt{P_{t-s}f+\epsilon})^2h(x)m(x)+2K\sum\limits_{x\in V}\Gamma(\sqrt{P_{t-s}f+\epsilon})h(x)m(x)
\end{aligned}
\right.
\end{equation*}
By definition of completeness, there exists a non-decreasing sequence $\{\eta_k\}_{k=0}^{\infty}$ defined as \eqref{eq:complete} on $G$. Then for any $0\leq h\in D(Q)$, we have $h\eta_k\in C_0(V)$ and when $k\rightarrow\infty$, we have $h\eta_k\rightarrow h$. Replace $h$ by $h\eta_k\in C_0(V)$ into above equality. Since $\Gamma(\sqrt{P_{t-s}f+\epsilon},\frac{\Delta( P_{t-s}f+\epsilon)}{\sqrt{P_{t-s}f+\epsilon}})\in l^1(V,m)$, $\Gamma(\sqrt{P_{t-s}f+\epsilon})\in l^1(V,m)$, $(P_{t-s}f+\epsilon)(\Delta\log\sqrt{P_{t-s}f+\epsilon})^2\in l^1(V,m)$ and $\Gamma(\sqrt{P_{t-s}f+\epsilon})\in D(Q)$, it is obvious that when  $k\rightarrow\infty$, we have
\begin{equation*}
\left.
\begin{aligned}
&\quad -2\sum\limits_{x\in V}\Gamma(\sqrt{P_{t-s}f+\epsilon},\frac{\Delta( P_{t-s}f+\epsilon)}{\sqrt{P_{t-s}f+\epsilon}})(x)h\eta_k(x)m(x)\\
&\rightarrow -2\sum\limits_{x\in V}\Gamma(\sqrt{P_{t-s}f+\epsilon},\frac{\Delta( P_{t-s}f+\epsilon)}{\sqrt{P_{t-s}f+\epsilon}})(x)h(x)m(x),\\
\end{aligned}
\right.
\end{equation*}
and
\begin{equation*}
\left.
\begin{aligned}
&2K\sum\limits_{x\in V}\Gamma(\sqrt{P_{t-s}f+\epsilon})h\eta_k(x)m(x)\\
&\rightarrow2K\sum\limits_{x\in V}\Gamma(\sqrt{P_{t-s}f+\epsilon})h(x)m(x),
\end{aligned}
\right.
\end{equation*}
and
\begin{equation*}
\left.
\begin{aligned}
&\sum\limits_{x\in V}(P_{t-s}f+\epsilon)(\Delta\log\sqrt{P_{t-s}f+\epsilon})^2h\eta_k(x)m(x)\\
&\rightarrow\sum\limits_{x\in V}(P_{t-s}f+\epsilon)(\Delta\log\sqrt{P_{t-s}f+\epsilon})^2h(x)m(x).
\end{aligned}
\right.
\end{equation*}
 Moreover, we have
\begin{equation*}
\begin{aligned}
&\quad\left|\sum\limits_{x\in V}\Gamma(\Gamma(\sqrt{P_tf+\epsilon}),h\eta_k)(x)m(x)-\sum\limits_{x\in V}\Gamma(\Gamma(\sqrt{P_tf+\epsilon}),h)(x)m(x)\right|\\
&=\left|\sum\limits_{x\in V}\Gamma(\Gamma(\sqrt{P_tf+\epsilon}),h\eta_k-h)(x)m(x)\right|\\
&\leq\sum\limits_{x\in V}\sqrt{\Gamma(\Gamma(\sqrt{P_tf+\epsilon})}(x)\sqrt{\Gamma(h(\eta_k-\mathbf{1}))}(x)m(x)\\
&\leq\left(\sum\limits_{x\in V}\Gamma(\Gamma(\sqrt{P_tf+\epsilon})(x)m(x)\right)^{1/2}\left(\sum\limits_{x\in V}\Gamma(h(\eta_k-\mathbf{1}))(x)m(x)\right)^{1/2}\\
&\rightarrow 0.
\end{aligned}
\end{equation*}
We use Lemma \ref{lemma1.2} in the last step. Then we finally conclude that for any $0\leq h\in D(Q)$ the claim holds.
\end{proof}

Since $P_s$ is self-adjoint operator on $\ell_m^2$, and by choosing the delta function, such as $\xi=\delta_y(x) (y\in V)$, then
\[\frac{\Phi(s,x)}{m(x)}=P_s\left(\Gamma(\sqrt{P_{t-s}f+\epsilon})\right)(x):=\phi(s,x),\]
\begin{theorem}\label{th:important}
Let $G=(V,E,m,\omega)$ be a complete graph and $m$ be a non-degenerate measure. If $G$ satisfies $CDE'(n,K)$, then for any smooth positive function $\alpha:[0,t]\rightarrow \mathbb{R}_{>0}$, and non-positive $\gamma:[0,t]\rightarrow \mathbb{R}_{\leq 0}$, for any function $0\leq f\in C_0(V)$, we have
\begin{equation}\label{eq:important1}
(\alpha \phi)'\geq\left(\alpha'+2\alpha K-\frac{4\alpha\gamma}{n}\right)\phi+\frac{2\alpha\gamma}{n}\Delta P_tf
-\frac{2\alpha\gamma^2}{n}(P_tf+\varepsilon)
\end{equation}

\end{theorem}
\begin{proof}
Since $P_s\xi \in D(Q)$, combining the equation \eqref{eq:derive} with Claim \ref{cl:1}, we then conclude that
\[\Phi'(s)\geq \frac{2}{n}\sum\limits_{x\in V}(P_{t-s}f+\epsilon)(\Delta\log\sqrt{P_{t-s}f+\epsilon})^2(P_s\xi)(x)m(x)+2K\sum\limits_{x\in V}\Gamma(\sqrt{P_{t-s}f+\epsilon})(P_s\xi)(x)m(x).\]
For any $x\in V$, we have
\begin{equation*}
\begin{aligned}
(\alpha\Phi)'(s)
&=\alpha'\Phi(s)+\alpha\Phi'(s)\\
&\geq\left(\alpha'+2\alpha K\right)\Phi(s)+\frac{2\alpha}{n}\sum\limits_{x\in V}(P_{t-s}f+\epsilon)(\Delta\log\sqrt{P_{t-s}f+\epsilon})^2(P_s\xi)(x)m(x)
\end{aligned}
\end{equation*}
Firstly, we separate the second item into two parts. Since for any $x\in V$, and any $0<g\in V^\mathbb{R},$ we have the following simple estimate
$g(x) \Delta \log g(x)\leq \Delta g(x).$
Then if $\Delta g(x)<0$, we have
\[g^2(x) (\Delta \log g)^2(x)\geq (\Delta g)^2(x).\]
Therefore,
\begin{equation*}
\begin{aligned}
&\sum\limits_{x\in V}(P_{t-s}f+\epsilon)(\Delta\log\sqrt{P_{t-s}f+\epsilon})^2(P_s\xi)(x)m(x)\\
&\geq\sum\limits_{x\in V,\Delta \sqrt{P_{t-s}f+\epsilon}(x)<0 }(\Delta\sqrt{P_{t-s}f+\epsilon})^2(P_s\xi)(x)m(x)\\
&+\sum\limits_{x\in V,\Delta \sqrt{P_{t-s}f+\epsilon}(x)\geq0 }(P_{t-s}f+\epsilon)(\Delta\log\sqrt{P_{t-s}f+\epsilon})^2(P_s\xi)(x)m(x).\\
\end{aligned}
\end{equation*}
Moreover, for any function $\gamma$, one has
\[(\Delta\sqrt{P_{t-s}f+\epsilon})^2\geq2\gamma\sqrt{P_{t-s}f+\epsilon}\Delta\sqrt{P_{t-s}f+\epsilon}-\gamma^2(P_{t-s}f+\epsilon),\]
since $\gamma$ is non-positive, if $\Delta \sqrt{P_{t-s}f+\epsilon}(x)\geq0$, it is also true that
\[(P_{t-s}f+\epsilon)(\Delta\log\sqrt{P_{t-s}f+\epsilon})^2\geq2\gamma\sqrt{P_{t-s}f+\epsilon}\Delta\sqrt{P_{t-s}f+\epsilon}-\gamma^2(P_{t-s}f+\epsilon),\]
as the right hand side of this inequality is clearly non-positive.
Furthermore, by \eqref{eq:solution}, one can conclude that
\begin{equation*}
\begin{aligned}
&\sum\limits_{x\in V}(P_{t-s}f+\epsilon)(\Delta\log\sqrt{P_{t-s}f+\epsilon})^2(P_s\xi)(x)m(x)\\
&\geq\sum\limits_{x\in V}\left[\gamma\Delta(P_{t-s}f+\epsilon)-2\gamma\Gamma(\sqrt{P_{t-s}f+\epsilon})-\gamma^2(P_{t-s}f+\epsilon)\right](P_s\xi)(x)m(x).
\end{aligned}
\end{equation*}
Therefore, we have
\begin{equation}\label{eq:important}
\begin{aligned}
(\alpha\Phi)'(s)&\geq\left(\alpha'+2\alpha K-\frac{4\alpha\gamma}{n}\right)\Phi(s)+\frac{2\alpha\gamma}{n}\sum\limits_{x\in V}\Delta(P_{t-s}f+\epsilon)(P_s\xi)(x)m(x)\\
&-\frac{2\alpha\gamma^2}{n}\sum\limits_{x\in V}(P_{t-s}f+\epsilon)(P_s\xi)(x)m(x).
\end{aligned}
\end{equation}
Let $\xi(x)=\delta_y(x)$ in the above inequality, by the self-adjoint property of $P_t$, we obtain
\begin{equation}\label{eq:im}
\begin{aligned}
\frac{d}{ds}(\alpha\phi)(s,y)&\geq\left(\alpha'+2\alpha K-\frac{4\alpha\gamma}{n}\right)\phi(s,y)+\frac{2\alpha\gamma}{n}P_s\Delta(P_{t-s}f+\epsilon)(y)\\
&-\frac{2\alpha\gamma^2}{n}P_s(P_{t-s}f+\epsilon)(y).
\end{aligned}
\end{equation}
According to the semigroup property of heat semigroup, i.e. $P_tP_sf=P_{t+s}f$, as well as the commutative property of $\Delta$ and $P_t$, that is, $\Delta P_tf=P_t\Delta f$, and the stochastic completeness for graphs in these assumption, i.e. $P_t\mathbf{1}=\mathbf{1}$. From \eqref{eq:im},
we can complete the proof.
\end{proof}

If $\alpha$ is chosen appropriately to make $\gamma$ non-positive, and satisfying
\[\alpha'+2\alpha K-\frac{4\alpha\gamma}{n}=0,\]
then we may integrate the inequality \eqref{eq:important1} in Theorem \ref{th:important}, if we denote $W(s)=\sqrt{\alpha}(s)$, then we obtain the following estimate.
\begin{theorem}
Let $G=(V,E,m,\omega)$ be a  complete graph and $m$ be a non-degenerate measure. If $G$ satisfies $CDE'(n,K)$, and $W:[0,t]\rightarrow  \mathbb{R}_{>0}$ be a smooth and positive function such that
\[W(0)=1,W(t)=0,\]
and  for $s\in [0,t]$,
\[W'(s)\leq-KW(s).\]
Then for any $0\leq f\in \ell^p_m$ with $p\in[1,\infty]$, we have
\begin{equation}\label{eq:ly-family}
\begin{aligned}
\frac{\Gamma(\sqrt{P_tf})}{P_tf}&\leq\frac{1}{2}\left(1-2K\int_0^tW(s)^2ds\right)\frac{\Delta P_tf}{P_tf}\\
&+\frac{n}{2}\left(\int_0^tW'(s)^2ds+K^2\int_0^tW(s)^2ds-K\right).
\end{aligned}
\end{equation}
\end{theorem}
\begin{proof}By integrating the inequality \eqref{eq:important1} in Theorem \ref{th:important}, we obtain for any $0\leq f\in C_0(V)$,
\begin{equation*}
\begin{aligned}
\Gamma(\sqrt{P_tf+\epsilon})&\leq\frac{1}{2}\left(1-2K\int_0^tW(s)^2ds\right)\Delta P_tf\\
&+\frac{n}{2}\left(\int_0^tW'(s)^2ds+K^2\int_0^tW(s)^2ds-K\right)(P_tf+\epsilon).
\end{aligned}
\end{equation*}
Notice that G is locally finite, then
\[\lim_{\epsilon\rightarrow 0}\Gamma(\sqrt{P_tf+\epsilon})=\Gamma(\sqrt{P_tf})\]
That means for any $0\leq f\in C_0(V)$, we have
\begin{equation}\label{eq:ly-finite}
\begin{split}
\Gamma(\sqrt{P_tf})&\leq\frac{1}{2}\left(1-2K\int_0^tW(s)^2ds\right)\Delta P_tf\\
&+\frac{n}{2}\left(\int_0^tW'(s)^2ds+K^2\int_0^tW(s)^2ds-K\right)(P_tf).
\end{split}
\end{equation}
For any $0\leq f\in \ell^p_m$ with $p\in[1,\infty]$, consider the sequence $\{f\eta_k\}_{k=0}^{\infty}$ while $\{\eta_k\}_{k=0}^{\infty}$ is the monotonically non-decreasing sequence about $k$ defined as \eqref{eq:complete}. It is obvious that $0\leq\eta_k\in C_0(V)$. Moreover, $0\leq f\eta_k\in C_0(V)$, and $f\eta_k\rightarrow f$ pointwise when $k\rightarrow\infty$ .

Applying \eqref{eq:ly-finite} to $f\eta_k$, we have
\begin{equation}\label{eq:ly-family3}
\begin{split}
\Gamma\left(\sqrt{P_t(f\eta_k)}\right)&\leq\frac{1}{2}\left(1-2K\int_0^tW(s)^2ds\right)\Delta P_t(f\eta_k)\\
&+\frac{n}{2}\left(\int_0^tW'(s)^2ds+K^2\int_0^tW(s)^2ds-K\right)P_t(f\eta_k).
\end{split}
\end{equation}
By monotone convergence theorem, $P_t(f\eta_k)\rightarrow P_tf$ pointwise when $k\rightarrow\infty$. Noting that $G$ is locally finite, we have
$$\lim\limits_{k\rightarrow\infty}\Gamma\left(\sqrt{P_t(f\eta_k)}\right)=\Gamma\left(\lim\limits_{k\rightarrow\infty}\sqrt{P_t(f\eta_k)}\right)=\Gamma(\sqrt{P_tf}),$$
and
$$\lim\limits_{k\rightarrow\infty}\Delta\left(P_t(f\eta_k)\right)=\Delta\left(\lim\limits_{k\rightarrow\infty}P_t(f\eta_k)\right)=\Delta(P_tf),$$
Therefore, let $k\rightarrow\infty$ in the both side of the inequality \eqref{eq:ly-family3}, we obtain what we desire.
\end{proof}
\begin{proof}[Proof of Theorem \ref{th:ly-family}]
As observed in \cite{BG09}, we choose that
\[W(s)=\left(1-\frac{s}{t}\right)^b,\]
for any $b >\frac{1}{2}$ is quite interesting in the regime where $-\frac{b}{t}<K$.
For this family
\[\int_0^tW(s)^2ds=\frac{t}{2b+1},\]
and
\[\int_0^tW'(s)^2ds=\frac{b^2}{(2b-1)t}.\]
Thus, for such a choice of $W$, the estimate \eqref{eq:ly-family} yields the family of Li-Yau inequalities \eqref{eq:ly-family2}. That finishes our proof.
\end{proof}

Next, we prove applications from Li-Yau inequality as we mentioned before.

\begin{proof}[Proof of Theorem \ref{th:Gauss-upper}]
For any $y\in V$, we let $f=\delta_y$ in \eqref{eq:hanack}. For any $t>0$, choosing $s=2t$ and let $z\in B(x,\sqrt{t})$, we have
\[p(t,x,y)
\leq p(2t,z,y)2^{n}\exp\left(\frac{4m_{max}}{\omega_{min}}\right).\]
Integrating the above inequality over $B(x,\sqrt{t})$ with respect to $z$, gives
\[
\begin{split}
p(t,x,y)
&\leq \frac{C}{V(x,\sqrt{t})}\sum_{z\in B(x,\sqrt{t})}m(z)p(2t,z,y)\\
&\leq \frac{C}{V(x,\sqrt{t})}.
\end{split}
\]
Where $C=2^{n}\exp\left(\frac{4m_{max}}{\omega_{min}}\right)$, and we also use the fact that $P_t\mathbf{1}=\mathbf{1}$.
\end{proof}
\begin{proof}[Proof of Theorem \ref{th:Cheng}]
By a result of S. Haeseler and M. Keller ([9], Theorem 3.1), if $\lambda\leq \lambda^*$, then there would be a positive solution $u$ to $\Delta u = -\lambda u.$
Moreover, for positive time-independent solutions to the equation $\Delta f_0 =-\lambda f_0$ with $f_0\in\ell^2_m\subset\ell^\infty_m$, then we have
 $\Delta P_t f_0= P_t\Delta f_0=-\lambda P_t f_0$,
and the Li-Yau inequality \eqref{eq:ly-family2} reduces to
\[
\frac{\Gamma(\sqrt{P_t f_0})}{P_t f_0}+\frac{1}{2}\left(1+\frac{2Kt}{2b+1}\right)\frac{\lambda P_t f_0}{P_t f_0}\leq\frac{n}{2}\left(\frac{b^2}{(2b-1)t}+\frac{K^2 t}{2b+1}+K\right).
\]
Noting that $\frac{\Gamma(\sqrt{P_t f_0})}{P_t f_0}\geq 0$, dividing $t$ on the both sides and taking the limit $t\rightarrow \infty,$ we conclude that
\[\lambda \leq \frac{Kn}{2}.\]
Since $\lambda$ is arbitrary, we can finish the proof.
\end{proof}

 {\bf Acknowledgements.} Y. Lin is supported by the National Science Foundation of China (Grant No.11271011), Shing-Tung Yau is supported by the NSF DMS-0804454. We thank Melchior Wirth for many valuable comments on the paper. Part of the work was done when Y. Lin visited S.-T. Yau in CMSA of Harvard University
in 2018.

\bibliographystyle{amsalpha}

Chao Gong,\\
Department of Mathematics, Renmin University of China, Beijing, China\\
\textsf{elfmetal@ruc.edu.cn}\\
Yong Lin,\\
Department of Mathematics, Renmin University of China, Beijing, China\\
\textsf{linyong01@ruc.edu.cn}\\
Shuang Liu,\\
Yau Mathematical Sciences Center, Tsinghua University, Beijing, China\\
\textsf{shuangliu@mail.tsinghua.edu.cn}\\
Shing-Tung Yau,\\
Department of Mathematics, Harvard University, Cambridge, Massachusetts, USA\\
\textsf{yau@math.harvard.edu}\\

\end{document}